\documentclass[11pt]{article}
\usepackage{latexsym}
\usepackage{amscd}
\usepackage[all]{xy}
\usepackage[mathscr]{euscript}

\usepackage[cp1251]{inputenc}
\usepackage[english]{babel}
\usepackage{amssymb,amsmath,amsthm,tikz,mathtools,mathrsfs}
\usepackage{graphicx,enumitem}
\usetikzlibrary{matrix,arrows,decorations.pathmorphing}

\theoremstyle{definition}
\newtheorem{theorem}{Theorem}[section]
\newtheorem{proposition}[theorem]{Proposition}

\newtheorem{remark}[theorem]{Remark}
\newtheorem{definition}[theorem]{Definition}

\newcommand{\mc}[1]{\mathcal{#1}}

\numberwithin{equation}{section}

\begin{document}

\title{A note on fixed points of endofunctors}

\author{Aleksandr Luzhenkov}

\date{}
\maketitle

\begin{abstract}
In this note, we deal with the fixed points of an endofunctor $F: \mc C \longrightarrow \mc C$. Three classes of fixed points are introduced, and the case when $F$ is an endomorphism of a category with pretopology is investigated. We show existence of induced structures on the category of fixed points, and, when a pretopology is defined, give a characterization of fixed points in terms of sheaf cohomology on $\mc C$.
\end{abstract}

\setcounter{page}{1}

\section{Introduction}

Let $\mc C$ be a category, $F: \mc C \longrightarrow \mc C$ a functor. An object $X \in \mc C$ is called a fixed point of $F$ if there is an isomorphism $X \simeq F(X)$. For the first time, fixed points of a functor appeared in \cite{l}, where, among other things, an existence theorem was established and morphism of fixed points was introduced.

In this note, we give a few results concerning fixed points of a functor. We deal with three types of fixed points; the strict ones, the usual fixed points, and the cohomological fixed points. While the first and the second type are defined in terms of an identity morphism and an isomorphism, respectively, the third type is defined for a site morphism $F: \mc C \longrightarrow \mc C$ in terms of sheaf cohomology functors $H^{i}(X, \cdot): Sh(\mc C) \longrightarrow Ab, X \in \mc C$. Every strict fixed point is a fixed point, and (when we are working with a site morphism) every fixed point is the cohomological one.

In chapter 2, we deal with fixed points and strict fixed points. We introduce connection between strict fixed points of a functor and the nerve of a category. In chapter 3, we show that the category of fixed points $\mc{S}(F)$ of an additive endofunctor $F$ of an additive category is an additive category; if a functor $F: \mc C \longrightarrow \mc C$ is a site morphism, its category of fixed points $\mc{S}(F)$ may be endowed with the structure of a site. In chapter 4, using (co)slice categories, we give a criterion for a pair $(X, X \overset{\sigma}{\longrightarrow} F(X))$ to be a fixed point. In chapter 5, we give characterization of fixed points in terms of sheaf cohomology on a site, and introduce cohomological fixed points of a site endomorphism.

\vspace{5mm}

\underline{Convention}. Throughout the text, we suppose that a universe $\mc U$ is fixed. An element of $\mc U$ is called a $\mc U$-small set (or just a \emph{small set}). Category $\mc C$ is called small if $Ob(\mc C)$ is a small set and $Hom_{\mc C}(X,Y)$ is a small set for every pair of objects $X,Y \in Ob(\mc C)$. A ``category $\mc C$'' is always meant to be a ``small category $\mc C$''.

\section{Fixed points of endofunctor}

Let $\mc C$ be a category, $F: \mc C \longrightarrow \mc C$ a functor.

\begin{definition}

A pair $(X, \alpha)$ of an object $X \in Ob(\mc C)$ and an isomorphism $\alpha: X \widetilde{\longrightarrow} F(X)$ in $\mc C$ is called a fixed point of a functor $F: \mc C \longrightarrow \mc C$.

\end{definition}

A morphism between fixed points $(X, \alpha)$ and $(Y, \beta)$ is a morphism $f: X \longrightarrow Y$ in $\mc C$ such that the diagram

\begin{equation}
\begin{gathered}
\xymatrix@R=3pc@C=3pc{
  X \ar[d]_{f} \ar[r]^{\alpha} & F(X) \ar[d]^{F(f)}  \\
  Y \ar[r]^{\beta} & F(Y)             }
\end{gathered}
\end{equation}
is commutative. We will denote it as $(X, \alpha) \overset{f}{\longrightarrow} (Y, \beta)$.

\begin{definition}

A category $\mc{S}(F)$ of fixed points of a functor $F: \mc{C} \longrightarrow \mc{C}$ has fixed points $(X, \alpha)$ as objects and morphisms between objects given by the diagram (2.1).

\end{definition}

For every $F: \mc C \longrightarrow \mc C$ there is a faithful functor $for: \mc{S}(F) \longrightarrow \mc{C}$ which is defined by the rule

\begin{center}

$(X, \alpha) \mapsto X$,

$((X, \alpha) \overset{f}{\longrightarrow} (Y, \beta)) \mapsto (X \overset{f}{\longrightarrow} Y)$.

\end{center}

A category of fixed points $\mc{S}(F)$ is determined by the functor $F$. It might be the case that $\mc{S}(F)$ is an empty category --- a category without objects and without morphisms, or it may be equivalent to $\mc{C}$. In some cases, there is always at least one fixed point, as happens when  $\mc{C}$ is abelian and $F: \mc{C} \longrightarrow \mc{C}$ is additive. It's also clear that if $(X, \alpha)$ is a fixed point of $F$ then for every object $Y$ from the isomorphism class of objects $[X]$ there is an isomorphism $Y \widetilde{\longrightarrow} F(Y)$

\begin{proposition}

Let $F_{i}, F_{j}: \mc{C} \longrightarrow \mc{C}$ be isomorphic functors. Then categories $\mc{S}(F_{i})$ and $\mc{S}(F_{j})$ are isomorphic.

\end{proposition}

\begin{proof}

Since $F_{i}$ and $F_{j}$ are isomorphic functors, there are natural transformations

\begin{center}

$\eta: F_{i} \longrightarrow F_{j}$,

$\eta^{-1}: F_{j} \longrightarrow F_{i}$,
\end{center}
such that $\eta^{-1} \circ \eta = Id_{F_{i}}, \eta \circ \eta^{-1} = Id_{F_{j}}$.

Assume that $(X_{1}, \alpha_{1})$ is a fixed point of $F_{i}$. Then we have an isomorphism

\begin{center}

$X_{1} \overset{\alpha_{1}}{\widetilde{\longrightarrow}} F_{i}(X_{1}) \overset{\eta_{X_{1}}}{\widetilde{\longrightarrow}} F_{j}(X_{1})$.

\end{center}

Therefore, for each fixed point $(X_{1}, \alpha_{1})$ of $\mc{S}(F_{i})$ we have a fixed point $(X_{1}, \eta_{X_{1}} \circ \alpha_{1})$ of $\mc{S}(F_{j})$.

Suppose there is a morphism of fixed points in $\mc{S}(F_{i})$ that is given by the diagram

\begin{equation}
\begin{gathered}
\xymatrix@R=3pc@C=3pc{
  X_{1} \ar[d]_{f} \ar[r]^{\alpha_{1}} & F_{i}(X_{1}) \ar[d]^{F_{i}(f)}  \\
  X_{2} \ar[r]^{\alpha_{2}} & F_{i}(X_{2})             }
\end{gathered}
\end{equation}

Then we have the morphism between $(X_{1}, \eta_{X_{1}} \circ \alpha_{1})$ and $(X_{2}, \eta_{X_{2}} \circ \alpha_{2})$, since the diagram

\begin{equation}
\begin{gathered}
\xymatrix@R=3pc@C=3pc{
  X_{1} \ar[d]_{f} \ar[r]^{\alpha_{1}} & F_{i}(X_{1}) \ar[d]^{F_{i}(f)} \ar[r]^{\eta_{X_{1}}} & F_{j}(X_{1}) \ar[d]^{F_{j}(f)}  \\
  X_{2} \ar[r]^{\alpha_{2}} & F_{i}(X_{2}) \ar[r]^{\eta_{X_{2}}} & F_{j}(X_{2})           }
\end{gathered}
\end{equation}
is commutative. We have a well-defined functor $(\eta)^{*}: \mc{S}(F_{i}) \longrightarrow \mc{S}(F_{j})$, given by

\begin{center}

$(X, \alpha) \mapsto (X, \eta_{X} \circ \alpha)$,

$((X_{1}, \alpha_{1}) \overset{f}{\longrightarrow} (X_{2}, \alpha_{2})) \mapsto ((X_{1}, \eta_{X_{1}} \circ \alpha_{1}) \overset{f}{\longrightarrow} (X_{2}, \eta_{X_{2}} \circ \alpha_{2}))$.
\end{center}
Functor $(\eta)^{*}$ has the inverse functor, namely $(\eta^{-1})^{*}$.

\end{proof}

Let us define a presheaf on $\mc{C}$ by the rule

\begin{equation*}
X \xmapsto{\hspace*{0.5cm}} \varinjlim_{X_{i} \in Ob(\mc{C})}Hom_{\mc{C}}(X,F(X_{i})).
\end{equation*}

\begin{proposition}
Let $F: \mc{C} \longrightarrow \mc{C}$ be a full functor, $(X,\alpha)$ a fixed point of $F$. There is a bijection
\begin{equation*}
\varinjlim_{X_{i} \in Ob(\mc{C})}Hom_{\mc{C}}(X,F(X_{i})) \simeq \{pt\},
\end{equation*}
where $\{pt\}$ is a one-point set.
\end{proposition}

\begin{proof}

Take an arbitrary morphism $\phi: X \longrightarrow F(X_{j})$ in $\mc{C}$. We are going to demonstrate that for a morphism $\phi$ there is always a commutative diagram

\begin{equation}
\begin{gathered}
\xymatrix@R=3pc@C=3pc{
  X \ar[dr]_{\alpha} \ar[rr]^{\phi} & & F(X_{j})  \\
  & F(X) \ar[ur]_{F(f)}  &     }
\end{gathered}
\end{equation}
for some $f: X \longrightarrow X_{j}$. Indeed, there is always a morphism $\phi \circ \alpha^{-1}: F(X) \longrightarrow F(X_{j})$ that makes the diagram above a commutative one. Since functor $F$ is full, $\phi \circ \alpha^{-1} = F(f)$ for some $f: X \longrightarrow X_{j}$.

On the other hand, existence of the diagram (2.4) is equivalent to the identity

\begin{equation}
\phi = Hom_{\mc{C}}(X,F(f))\alpha,
\end{equation} for some $f: X \longrightarrow X_{j}$. Therefore, the inductive limit $\varinjlim_{X_{i} \in Ob(\mc{C})}Hom_{\mc{C}}(X,F(X_{i}))$ is in bijection with a one-point set.

\end{proof}

Now we introduce a subclass of fixed points of the functor $F: \mc{C} \longrightarrow \mc{C}$.

\begin{definition}
An object $X \in Ob(\mc{C})$ is called a strict fixed point of the functor $F: \mc{C} \longrightarrow \mc{C}$ if $X = F(X)$.
\end{definition}

A strict fixed point of the functor $F$ is a fixed point of the form $(X, Id)$. In order to establish connection between strict fixed points of a functor and fixed points of continuous maps between topological spaces, we briefly recall the construction of the nerve of a small category (we follow \cite{gm}; all the details may be found there).

Let $\Delta$ be the category of finite sets and nondecreasing maps between them. A simplicial set $X$ is a functor $X: \Delta^{op} \longrightarrow Set$. We denote $X_{n} \overset{def}{=} X([n])$. A simplicial map between two simplicial sets $X$ and $X'$ is just a morphism of functors. The category of simplicial sets is denoted as $sSet$.

With every $X \in Ob(sSet)$ we associate the set

\begin{equation}
\coprod_{n=0}^{\infty}(\Delta_{n} \times X_{n})/ \sim,
\end{equation}
where $\Delta_{n}$ is usual topological $n$-simplex and $(s,x) \in \Delta_{n} \times X_{n} \sim (t,y) \in \Delta_{m} \times X_{m}$ if and only if $y = X(f)x, s = \Delta_{f}(t)$, for some nondecreasing map $f: [m] \longrightarrow [n]$. Here, $\Delta_{f}$ is the linear map $\Delta_{m} \longrightarrow \Delta_{n}$ that maps the vertex $v_{i} \in \Delta_{m}$ to the vertex $v_{f(i)} \in \Delta_{n}, i = 0,...,m$. The geometric realization $|X|$ of a simplicial set $X$ is the set (2.6) endowed with the weakest topology for which the factorization by $\sim$ is continuous.

A simplicial map $F: X \longrightarrow X'$ induces a continuous map $|F|: |X| \longrightarrow |X'|$. We introduce the map $F^{*}: \coprod_{n=0}^{\infty}(\Delta_{n} \times X_{n}) \longrightarrow \coprod_{n=0}^{\infty}(\Delta_{n} \times X'_{n})$ by the rule $F^{*}(s,x) = (s, F_{n}(x))$. Since $F^{*}$ maps equivalent points to equivalent ones, it induces a continuous map $|F|: |X| \longrightarrow |X'|$.

We define $N\mc{C}$ to be a simplicial set such that

$$(N\mc{C})_{n} = \{diagrams~ of~ the~ form~ X_{0} \overset{\phi_{0}}{\longrightarrow} X_{1} \overset{\phi_{1}}{\longrightarrow} ... \overset{\phi_{n-1}}{\longrightarrow} X_{n}, X_{i} \in Ob(\mc{C}), \phi_{i} \in Mor(\mc{C})\}.$$

The map $f^{*}: (N\mc{C})_{n} \longrightarrow (N\mc{C})_{m}$ which corresponds to $f: [m] \longrightarrow [n]$ is defined by the rule $f^{*}(X_{0} \overset{\phi_{0}}{\longrightarrow} X_{1} \overset{\phi_{1}}{\longrightarrow} ... \overset{\phi_{n-1}}{\longrightarrow} X_{n}) = (Y_{0} \overset{\psi_{0}}{\longrightarrow} Y_{1} \overset{\psi_{1}}{\longrightarrow} ... \overset{\psi_{n-1}}{\longrightarrow} Y_{n})$, where $Y_{i} = X_{f(i)}, \psi_{i} = Id$ if $f(i) = f(i+1)$ and $\psi_{i} = \psi_{f(i+1)-1} \circ ... \circ \psi_{f(i)}$ otherwise. If $F: \mc{C} \longrightarrow \mc{C'}$ is a functor then $NF: N\mc{C} \longrightarrow N\mc{C'}$ is the morphism of simplicial sets that maps a simplex $(X_{i},\phi_{i})$ into the simplex $(F(X_{i}),F(\phi_{i}))$. The geometric realization $|N\mc{C}|$ of simplicial set is called the nerve of the category $\mc{C}$. Along with triangulable topological space $|N\mc{C}|$, there is a continuous map $|NF|: |N\mc{C}| \longrightarrow |N\mc{C'}|$ as soon as there is a functor $F: \mc{C} \longrightarrow \mc{C'}$.

\begin{proposition}

Let $F: \mc{C} \longrightarrow \mc{C}$ be a functor. If $|NF|: |N\mc{C}| \longrightarrow |N\mc{C}|$ has a fixed point then functor $F$ has a strict fixed point.
\end{proposition}

\begin{proof}
Points of topological space $|N\mc{C}|$ have the form $p = ([s_{0},Y_{0}], [s_{1},Y_{1}], ... )$, where $(s_{i}, Y_{i}) \in \Delta_{i} \times X_{i}$ and equivalence classes are given by the relation $(s_{i}, Y_{i}) \sim (s_{j}, Y_{j}) \Longleftrightarrow \Delta_{f}(s_{j}) = s_{i}, Y_{j} = X(f)Y_{i}$, for some $f: [j] \longrightarrow [i]$.

If $|NF|(p) = p$ then, in particular, $[s_{0},Y_{0}] = [s_{0},F(Y_{0})]$. It means that $(s_{0},Y_{0}) \sim (s_{0},F(Y_{0}))$, which, in turn, means that there is a nondecreasing map $f: [0] \longrightarrow [0]$ such that $X(f)(Y_{0}) = F(Y_{0})$. But there is only one map from $[0]$ to $[0]$ and it's $Id$. Since $X(Id) = Id$, we arrive to the equality $Y_{0} = F(Y_{0})$.

\end{proof}

Suppose that $\mc{C}$ is compact category, i.e. $|N\mc{C}|$ is compact space. In this case, we may transfer some basic facts from topology.

\begin{proposition}

If $\mc{C}$ has the initial object there is a strict fixed point of $F: \mc{C} \longrightarrow \mc{C}$.
\end{proposition}
\begin{proof}
Since $\mc{C}$ has the initial object, the space $|N\mc{C}|$ is contractible. By Lefschetz Fixed Point Theorem, $|NF|: |N\mc{C}| \longrightarrow |N\mc{C}|$ has a fixed point. Therefore, by the Proposition 2.6., there is a strict fixed point of $F$.
\end{proof}

\begin{proposition}
Let $F,F': \mc{C} \longrightarrow \mc{C}$ be two functors, $\phi: F \longrightarrow F'$ a natural transformation. If $|NF|$ has a fixed point, then $F'$ has a strict fixed point.
\end{proposition}
\begin{proof}
Since there is a natural transformation between functors, maps $|NF|$ and $|NF'|$ are homotopic. From the hypothesis and Lefschetz Fixed Point Theorem follows that $|NF'|$ has a fixed point. Therefore, by Proposition 2.6., $F'$ has a strict fixed point.
\end{proof}
\section{Induced structures on the category of fixed points}

If a category $\mc C$ has a certain structure, then, in some cases, the same structure may be defined on $\mc{S}(F)$ by using the functor $for: \mc{S}(F) \longrightarrow \mc C$.

If $\mc C$ is additive, or a site, then the same structures may be defined on $\mc{S}(F)$.

\begin{proposition}
Let $\mc C$ be an additive category and $F: \mc C \longrightarrow \mc C$ an additive functor. Then, category of fixed points $\mc{S}(F)$ is additive.
\end{proposition}

\begin{proof}

\underline{1}. A functor $F: \mc C \longrightarrow \mc C$ is additive if $F(0) \simeq 0$ and the canonical map $\sigma_{XY}: F(X \oplus Y) \longrightarrow F(X) \oplus F(Y)$ is an isomorphism for all $X,Y \in Ob(\mc{C})$. It means that $\mc{S}(F)$ contains zero object\footnote{Although, existence of zero object would follow from the statements 2-4 below.} $(0, \alpha_{0})$, where $\alpha_{0}$ is an isomorphism $0 \tilde{\longrightarrow} F(0)$.

\underline{2}. We should verify that $Hom_{\mc{S}(F)}((X, \alpha), (Y, \beta))$ is an abelian group. Suppose that $f,g: (X, \alpha) \longrightarrow (Y, \beta)$ are morphisms of fixed points

\begin{equation}
\begin{gathered}
\xymatrix@R=3pc@C=3pc{
  X \ar@<1ex>[d]^{g} \ar@<-1ex>[d]_{f} \ar[r]^-{\alpha} & F(X) \ar@<1ex>[d]^{F(g)} \ar@<-1ex>[d]_{F(f)}  \\
  Y \ar[r]^{\beta} & F(Y)          }
\end{gathered}
\end{equation}

Since category $\mc{C}$ is additive and so is the functor $F: \mc{C} \longrightarrow \mc{C}$, we have equalities

\begin{equation}
\beta \circ (f + g) = \beta \circ f + \beta \circ g,
\end{equation}
\begin{equation}
F(f + g) = F(f) + F(g).
\end{equation}
Taking into account that $f$ and $g$ are morphisms of fixed points, we see from (3.2) and (3.3) that the diagram

\begin{equation}
\begin{gathered}
\xymatrix@R=3pc@C=3pc{
  X \ar[d]_{f+g}  \ar[r]^-{\alpha} & F(X) \ar[d]^{F(f+g)}  \\
  Y \ar[r]^{\beta} & F(Y)          }
\end{gathered}
\end{equation}
 commutes. Hence, $f+g$ is a morphism of fixed points. This sum is bilinear and associative, since so is the sum of morphisms in $\mc{C}$.

\underline{3}. We have to show that for every $f \in Hom_{\mc{S}(F)}((X, \alpha), (Y, \beta))$ there exists $g \in Hom_{\mc{S}(F)}((X, \alpha), (Y, \beta))$ that is the inverse element for $f$.

First, there is always a zero morphism $0_{XY}$ between fixed points $(X, \alpha)$ and $(Y, \beta)$. Indeed, the identity $F(0_{XY}) \circ \alpha = \beta \circ 0_{XY}$ obviously holds. We take $g$ to be inverse to $f$ in $\mc{C}$. It remains to show that $g$ is a morphism of fixed points.

Since $g$ is inverse to $f$ in $\mc{C}$, there are two identities $\beta \circ f + \beta \circ g = 0_{XF(Y)}, F(f) \circ \alpha + F(g) \circ \alpha = 0_{XF(Y)}$. Therefore
\begin{equation}
\beta \circ f + \beta \circ g = F(f) \circ \alpha + F(g) \circ \alpha.
\end{equation}
Morphism $f$ is a fixed points morphism, which means that (3.5) is equivalent to the identity
$$\beta \circ f + \beta \circ g = \beta \circ f + F(g) \circ \alpha.$$

After all, we arrive to the relation $\beta \circ g = F(g) \circ \alpha$, and $g$ is a fixed points morphism. Hence, $Hom_{\mc{S}(F)}((X, \alpha), (Y, \beta))$ is an abelian group.

\underline{4}. It remains to demonstrate that $\mc{S}(F)$ admits sums $(X, \alpha) \oplus (Y, \beta)$.  Since $F$ is an additive functor, the canonical map $F(X \oplus Y) \overset{\sigma_{XY}}{\longrightarrow} F(X) \oplus F(Y)$ is an isomorphism. On the other hand, the canonical morphism $X \oplus Y \overset{\alpha \oplus \beta}{\longrightarrow} F(X) \oplus F(Y)$, which is the universal morphism for the coproduct in $\mc{C}$, is an isomorphism. Then, we set

$$(X, \alpha) \oplus (Y, \beta) \overset{def}{=}  (X \oplus Y, \sigma_{XY}^{-1} \circ (\alpha \oplus \beta)).$$

Defined sum $(X, \alpha) \oplus (Y, \beta)$ is indeed a coproduct in $\mc{S}(F)$; universal property holds, and all morphisms are fixed points morphisms (for the details, see below the construction of pullback in $\mc{S}(F)$; arguments here and there are the same).

\end{proof}

We're going to show that category of fixed points of a site\footnote{What we call a site was originally called a category with pretopology.} morphism has the induced structure of a site. Before doing so, we have to set two definitions.

\begin{definition}

A category $\mc{C}$\footnote{Hereinafter, we suppose that $\mc{C}$ has pullbacks.} is called a site if for every object $X \in \mc{C}$ there is family of sets of morphisms into $X$, which is denoted as $cov(X)$. Three axioms should hold.
\begin{enumerate}[label=(\roman*)]
\item  For every isomorphism $f: X_{i} \longrightarrow X$, the set $\{\overset{f}{X_{i} \longrightarrow X}\}$ is in $cov(X)$.

\item Suppose that $\sigma: Y \longrightarrow X$ is an arbitrary arrow in $\mc{C}$, and a set $\{X_{i} \longrightarrow X\}_{i \in A}$ is in $cov(X)$. Then, the set of morphisms $\{X_{i} \times_{X} Y \longrightarrow Y\}_{i \in A}$ is in $cov(Y)$.

\item If $\{X_{i} \longrightarrow X\}_{i \in A} \in cov(X)$ and $\{X_{ij} \longrightarrow X_{i}\}_{j \in B(i)} \in cov(X_{i})$, for all $i \in A$, then the set $\{X_{ij} \longrightarrow X_{i} \longrightarrow X\}_{i \in A, j \in B(i)}$ is an element of $cov(X)$.
\end{enumerate}
\end{definition}

\begin{definition}
Let $\mc{C}$ be a site. A functor $F: \mc{C} \longrightarrow \mc{C}$ is called a morphism of sites if two axioms hold.
\begin{enumerate}[label=(\roman*)]
\item For every $\mc{U} \in cov(X)$ the set $F(\mc{U})$ is the element of $cov(F(X))$, for all $X \in Ob(\mc{C})$.

\item Functor $F$ commutes with pullbacks, i.e. the canonical map

$$F(X_{1} \times_{X} X_{2}) \longrightarrow F(X_{1}) \times_{F(X)} F(X_{2})$$

is an isomorphism, for all $X, X_{1}, X_{2} \in Ob(\mc{C})$.
\end{enumerate}

\end{definition}

First of all, we define pullbacks in $\mc{S}(F)$. Suppose that $(X_{1}, \alpha_{1}), (X_{2}, \alpha_{2}), (X_{3}, \alpha_{3})$ are fixed points of $F$ and $f: (X_{2}, \alpha_{2}) \longrightarrow (X_{1}, \alpha_{1}), g: (X_{3}, \alpha_{3}) \longrightarrow (X_{1}, \alpha_{1})$ are fixed points morphisms. Since $\mc{C}$ has pullbacks, there is a commutative diagram

\begin{equation}
\begin{gathered}
\xymatrix@R=3pc@C=3pc{
  X_{2} \times_{X_{1}} X_{3} \ar[d]^-{p^{*}_{f}(g)} \ar[r]^-{p^{*}_{g}(f)} & X_{3} \ar[d]^{g}  \\
  X_{2} \ar[r]^{f} & X_{1}          }
\end{gathered}
\end{equation}

We'll show that A) there exists isomorphism $\sigma$ such that $(X_{2} \times_{X_{1}} X_{3}, \sigma)$ is a fixed point, B) morphism $p^{*}_{f}(g)$ is actually a fixed points morphism $p^{*}_{f}(g): (X_{2} \times_{X_{1}} X_{3}, \sigma) \longrightarrow (X_{2}, \alpha_{2})$ and morphism $p^{*}_{g}(f)$ is also a fixed points morphism, C) universal property holds.

\underline{A}. Let us consider the pullback diagram

\begin{equation}
\begin{gathered}
\xymatrix@R=3pc@C=3pc{
  F(X_{2}) \times_{F(X_{1})} F(X_{3}) \ar[d]^{g'} \ar[r]^-{f'} & F(X_{3}) \ar[d]^{F(g)}  \\
  F(X_{2}) \ar[r]^{F(f)} & F(X_{1})          }
\end{gathered}
\end{equation}

From the universal property for $X_{2} \times_{X_{1}} X_{3}$, there is the unique morphism, which we denote $\sigma_{1}^{-1}$, that makes the diagram
\begin{equation}
\begin{gathered}
\xymatrix@R=3pc@C=3pc{
  F(X_{2}) \times_{F(X_{1})} F(X_{3}) \ar@/^2.5pc/[drr]^{\alpha_{3}^{-1} \circ f'} \ar@/_1pc/[ddr]_-{\alpha_{2}^{-1} \circ g'} \ar@{.>}[dr]^-{\sigma_{1}^{-1}} & &   \\
  & X_{2} \times_{X_{1}} X_{3} \ar[r]^-{p^*_{g}(f)} \ar[d]^{p^*_{f}(g)} &  X_{3} \ar[d]_{g} \\
  & X_{2} \ar[r]^{f} & X_{1}                                                                     }
\end{gathered}
\end{equation}
commute (note that $f,g$ are fixed points morphisms and diagram (3.7) is commutative). From the universal property for $F(X_{2}) \times_{F(X_{1})} F(X_{3})$ follows the existence of the unique morphism $\sigma_{1}: X_{2} \times_{X_{1}} X_{3} \longrightarrow F(X_{2}) \times_{F(X_{1})} F(X_{3})$ that makes the appropriate diagram commute. Exploiting again universal property of pullback, we see that $\sigma_{1}^{-1} \circ \sigma_{1} = Id_{X_{2} \times_{X_{1}} X_{3}}$ and $\sigma_{1} \circ \sigma_{1}^{-1} = Id_{F(X_{2}) \times_{F(X_{1})} F(X_{3})}$.

Composing $\sigma_{1}$ with canonical isomorphism $\sigma_{2}^{-1}: F(X_{2}) \times_{F(X_{1})} F(X_{3}) \longrightarrow F(X_{2} \times_{X_{1}} X_{3})$, we get the desired isomorphism $\sigma = \sigma_{2}^{-1} \circ \sigma_{1} : X_{2} \times_{X_{1}} X_{3} \longrightarrow F(X_{2} \times_{X_{1}} X_{3})$.

\underline{B}. Here, we're going to demonstrate that $p^*_{f}(g)$ and $p^*_{g}(f)$ are morphisms of fixed points. Indeed, from the above constructions and the definition of canonical morphism $\sigma_{2}$ follows that the diagram

\begin{equation}
\begin{gathered}
\xymatrix@R=3pc@C=3pc{
  F(X_{2} \times_{X_{1}} X_{3}) \ar[dr]^{F(p^*_{f}(g))} & \\
  F(X_{2}) \times_{F(X_{1})} F(X_{3}) \ar[u]^{\sigma^{-1}_{2}} \ar[r]^-{g'} & F(X_{2})   \\
  X_{2} \times_{X_{1}} X_{3} \ar[r]^{p^*_{f}(g)} \ar[u]^{\sigma_{1}} & X_{2} \ar[u]^{\alpha_{2}}         }
\end{gathered}
\end{equation}
is a commutative one. Therefore, $p^*_{f}(g)$ is a fixed points morphism.

From the same reasons follows that $p^*_{g}(f)$ is also a morphism of fixed points.

\underline{C}. It remains to prove the universal property. Suppose that we are given a fixed point $(X', \alpha')$ and two morphisms of fixed points $(X', \alpha') \overset{k}{\longrightarrow} (X_{3}, \alpha_{3}), (X', \alpha') \overset{h}{\longrightarrow} (X_{2}, \alpha_{2})$ such that $g \circ k = f \circ h$. By the universal property of pullback in $\mc{C}$, there is the morphism $\Phi: X' \longrightarrow X_{2} \times_{X_{1}} X_{3}$, such that $p^*_{f}(g) \circ \Phi = h$ and $p^*_{g}(f) \circ \Phi = k$. We need to show that $\Phi$ is a morphism of fixed points; in other words, the rectangular in the diagram

\begin{equation}
\begin{gathered}
\xymatrix@R=3pc@C=3pc{
  & X' \ar@/_2.5pc/[ddl]^{h} \ar@/^1.5pc/@{.>}[ddr]^{k} \ar[d]^{\Phi} \ar[rr]^{\alpha'} & & F(X') \ar[d]^{F(\Phi)}       \\
  & X_{2} \times_{X_{1}} X_{3} \ar[dl]_{p^*_{f}(g)} \ar[dr]^{p^*_{g}(f)} \ar[rr]^{\sigma}  & & F(X_{2} \times_{X_{1}} X_{3})     \\
  X_{2} \ar[r]^{f} & X_{1} & X_{3} \ar[l]_{g} &          }
\end{gathered}
\end{equation}
should be commutative. To demonstrate that this is the case, we'll use universal property of pullback in $\mc{C}$. There is the unique morphism $\phi: F(X') \longrightarrow X_{2} \times_{X_{1}} X_{3}$ that makes the diagram below a commutative one

\begin{equation}
\begin{gathered}
\xymatrix@R=2pc@C=2pc{
  F(X') \ar@/^2.5pc/[drr]^{\alpha_{3}^{-1} \circ F(k)} \ar@/_1pc/[ddr]_-{\alpha_{2}^{-1} \circ F(h)} \ar@{.>}[dr]^-{\phi} & &   \\
  & X_{2} \times_{X_{1}} X_{3} \ar[r]^-{p^*_{g}(f)} \ar[d]^{p^*_{f}(g)} &  X_{3} \ar[d]_{g} \\
  & X_{2} \ar[r]^{f} & X_{1}                                                                     }
\end{gathered}
\end{equation}

First, we'll demonstrate that $\Phi \circ (\alpha')^{-1}: F(X') \longrightarrow X_{2} \times_{X_{1}} X_{3}$ makes the diagram (3.11) a commutative one. Indeed, we have the sequence of identities $p^*_{f}(g) \circ \Phi = h \Longleftrightarrow p^*_{f}(g) \circ \Phi \circ (\alpha')^{-1} = h \circ (\alpha')^{-1} \Longleftrightarrow p^*_{f}(g) \circ \Phi \circ (\alpha')^{-1} = \alpha_{2}^{-1} \circ F(h)$. The last implication holds since $h$ is the morphism of fixed points. By the same reasons, noting that $k$ is the morphism of fixed points, we have the identity $p^*_{g}(f) \circ \Phi \circ (\alpha')^{-1} = \alpha_{3}^{-1} \circ F(k)$. Hence, $\phi = \Phi \circ (\alpha')^{-1}$.

Now, we're going to show that $\sigma^{-1} \circ F(\Phi)$ makes the diagram (3.11) commute. We have the sequence of identities $F(p^*_{f}(g)) \circ F(\Phi) = F(h) \Longleftrightarrow \alpha_{2}^{-1} \circ F(p^*_{f}(g)) \circ F(\Phi) = \alpha_{2}^{-1} \circ F(h) \Longleftrightarrow p^*_{f}(g) \circ \sigma^{-1} \circ F(\Phi) = \alpha_{2}^{-1} \circ F(h)$. The last implication holds since $p^*_{f}(g)$ is the morphism of fixed points. Again, noting that $p^*_{g}(f)$ is the morphism of fixed points and repeating the same arguments, we arrive to the equality $p^*_{g}(f) \circ \sigma^{-1} \circ F(\Phi) = \alpha_{3}^{-1} \circ F(k)$. Therefore, by the uniqueness of the universal morphism, $\phi = \sigma^{-1} \circ F(\Phi)$.

After all, we have that $\Phi \circ (\alpha')^{-1} = \sigma^{-1} \circ F(\Phi)$, and $\Phi$ is the morphism of fixed points.

\vspace{5mm}

By using functor $for: \mc{S}(F) \longrightarrow \mc{C}$ it's possible to define the structure of a site on $\mc{S}(F)$. A set of fixed points morphisms $\{(X_{i}, \alpha_{i}) \longrightarrow (X, \alpha)\}_{i \in A}$ is in $cov((X, \alpha))$ if $for(\{(X_{i}, \alpha_{i}) \longrightarrow (X, \alpha)\}_{i \in A})$ is in $cov(X)$. It's clear that all three axioms hold.
\section{Fixed points and (co)slice categories}

In this section, we want to give conditions under which a pair $(X, X \overset{\sigma}{\longrightarrow} F(X))$ is a fixed point by using comma categories $\mc{C}/X$ and $X/\mc{C}$; in particular, we'll give a criterion for $(X, X \overset{\sigma}{\longrightarrow} F(X))$ to be a fixed point. Throughout this section, we suppose that $\mc{C}$ has pullbacks and pushouts.

If we are given a morphism $\sigma: X \longrightarrow Y$ in $\mc{C}$ then it's possible to associate with it the functor

\begin{center}

$\tau^{-1}(\sigma): \tau^{-1}(Y) \longrightarrow \tau^{-1}(X)$,

\end{center}
where $\tau^{-1}(Z) \overset{def}{=} \mc{C}/Z$. Let us briefly recall this construction.

For every $g_{1}: Y_{1} \longrightarrow Y$ there is a commutative square

\begin{equation}
\begin{gathered}
\xymatrix@R=3pc@C=3pc{
  X \times_{Y} Y_{1} \ar[d]_{p^{*}_{\sigma}(g_{1})} \ar[r]^-{p^{*}_{g_{1}}(\sigma)} & Y_{1} \ar[d]^{g_{1}}  \\
  X \ar[r]^{\sigma} & Y          }
\end{gathered}
\end{equation}

Then, assuming the axiom of choice, we define the map $\tau^{-1}(\sigma): Ob(\mc{C}/Y) \longrightarrow Ob(\mc{C}/X)$ by the rule

\begin{center}

$(Y_{1} \overset{g_{1}}{\longrightarrow} Y) \mapsto (X \times_{Y} Y_{1} \overset{p^{*}_{\sigma}(g_{1})}{\xrightarrow{\hspace*{1cm}}} X)$.

\end{center}

Let $\eta_{12}$ be a morphism between $Y_{1} \overset{g_{1}}{\longrightarrow} Y$ and $Y_{2} \overset{g_{2}}{\longrightarrow} Y$ that is given by the commutative triangle

\begin{equation}
\begin{gathered}
\xymatrix@R=3pc@C=3pc{
  Y_{1} \ar[d]_{\eta_{12}} \ar[r]^{g_{1}} & Y  \\
  Y_{2} \ar[ur]_{g_{2}}   &     }
\end{gathered}
\end{equation}

We have to define a morphism $\eta_{12}^{*}: X \times_{Y} Y_{1} \longrightarrow X \times_{Y} Y_{2}$ such that the diagram

\begin{equation}
\begin{gathered}
\xymatrix@R=3pc@C=3pc{
  X \times_{Y} Y_{1} \ar[d]_{\eta_{12}^{*}} \ar[r]^-{p^*_{\sigma}(g_{1})} & X  \\
  X \times_{Y} Y_{2} \ar[ur]_{p^*_{\sigma}(g_{2})}   &     }
\end{gathered}
\end{equation}
is commutative.

Morphism $\eta_{12}^{*}$ will be extracted from the universal property of pullback. Indeed, we have the commutative diagram

\begin{equation}
\begin{gathered}
\xymatrix@R=3pc@C=3pc{
  X \times_{Y} Y_{1}  \ar@{.>}[dr]_{\Psi} \ar[rr]^{p^*_{g_1}(\sigma)} \ar@/_1pc/[ddr]_{p^*_{\sigma}(g_1)} & {} & Y_{1} \ar[dr]^{\eta_{12}} \ar@/^4pc/[ddr]^{g_1} & {}  \\
   & X \times_{Y} Y_{2} \ar[rr]^{p^*_{g_2}(\sigma)} \ar[d]^{p^*_{\sigma}(g_{2})}  & & Y_{2}   \ar[d]_{g_{2}}   \\
   & X \ar[rr]^{\sigma}   &   & Y  }
\end{gathered}
\end{equation}
where $\Psi$ is the universal morphism. We take $\eta^*_{12} = \Psi$.

Functor $\tau^{-1}(\sigma): \tau^{-1}(Y) \longrightarrow \tau^{-1}(X)$, therefore, is defined on objects and morphisms by the rule

\begin{center}

$(Y_{1} \overset{g_{1}}{\longrightarrow} Y) \mapsto (X \times_{Y} Y_{1} \overset{p^{*}_{\sigma}(g_{1})}{\xrightarrow{\hspace*{1cm}}} X)$,

\vspace{5mm}

$\begin{gathered}
\xymatrix@R=3pc@C=3pc{
  Y_{1} \ar[d]_{\eta_{12}} \ar[r]^{g_{1}} & Y  \\
  Y_{2} \ar[ur]_{g_{2}}   &     }
\end{gathered}$ $\mapsto$ $\begin{gathered}
\xymatrix@R=3pc@C=3pc{
  X \times_{Y} Y_{1} \ar[d]_{\eta_{12}^{*}} \ar[r]^-{p^*_{\sigma}(g_{1})} & X  \\
  X \times_{Y} Y_{2} \ar[ur]_{p^*_{\sigma}(g_{2})}   &     }
\end{gathered}$

\end{center}

Axioms of an identity morphism and composition are verified by the universal property of pullback; here, we omit details.

Of course, there is a dual construction: over-categories $\mc{C}/X$ are replaced by under categories $X/\mc{C}$, pullbacks are replaced by pushouts. Hence, for every morphism $\sigma: X \longrightarrow Y$ in $\mc{C}$ there is a well-defined functor $s^{-1}(\sigma): s^{-1}(X) \longrightarrow s^{-1}(Y)$, where $s^{-1}(Z) \overset{def}{=} Z/\mc{C}$.

\begin{proposition}

If $\sigma: X \longrightarrow Y$ is an isomorphism, functors

\begin{center}

$\tau^{-1}(\sigma): \tau^{-1}(Y) \longrightarrow \tau^{-1}(X)$

$s^{-1}(\sigma): s^{-1}(X) \longrightarrow s^{-1}(Y)$\end{center}
are equivalences of categories

\end{proposition}

\begin{proof}
Since pullbacks preserve isomorphisms (but not identities) and pasting-lemma holds, functor $\tau^{-1}(\sigma)$ has the quasi-inverse, namely the functor $\tau^{-1}(\sigma^{-1})$. The same argument applies to pushouts and functor $s^{-1}(\sigma)$.

\end{proof}

Besides functors $\tau^{-1}(\sigma)$ and $s^{-1}(\sigma)$, there are also two functors for every morphism $\newline
\sigma: X \longrightarrow Y$ in $\mc{C}$. The first one is $\overline{\sigma}: \tau^{-1}(X) \longrightarrow \tau^{-1}(Y)$, which is defined on objects by the post-composition
\begin{center}
$(X_{1} \overset{f_{1}}{\longrightarrow} X) \mapsto (X_{1} \overset{f_{1}}{\longrightarrow} X \overset{\sigma}{\longrightarrow} Y)$,
\end{center}
and the second one is $\tilde{\sigma}: s^{-1}(Y) \longrightarrow s^{-1}(X)$, which is defined on objects by the pre-composition
\begin{center}
$(Y \overset{g_{1}}{\longrightarrow} Y_{1}) \mapsto (X \overset{\sigma}{\longrightarrow} Y \overset{g_{1}}{\longrightarrow} Y_{1})$.
\end{center}

The following result is well-known.

\begin{proposition}

Let $\sigma: X \longrightarrow Y$ be a morphism in $\mc{C}$.

\begin{enumerate}[label=(\roman*)]
\item Functor $\tau^{-1}(\sigma)$ is right adjoint to the functor $\overline{\sigma}$.
\item Functor $s^{-1}(\sigma)$ is left adjoint to the functor $\tilde{\sigma}$.
\end{enumerate}

\end{proposition}

\begin{proof}

(i). Let $X_{1} \overset{f_{1}}{\longrightarrow} X$ be an object of $\tau^{-1}(X)$, $Y_{1} \overset{g_{1}}{\longrightarrow} Y$ an object of $\tau^{-1}(Y)$. We are going to construct maps between $Hom$-sets. First, we construct

\begin{equation*}
Hom_{\tau^{-1}(X)}(X_{1} \overset{f_{1}}{\longrightarrow} X, \tau^{-1}(\sigma)(Y_{1} \overset{g_{1}}{\longrightarrow} Y)) \overset{\beta}{\longrightarrow} Hom_{\tau^{-1}(Y)}(\overline{\sigma}(X_{1} \overset{f_{1}}{\longrightarrow} X), Y_{1} \overset{g_{1}}{\longrightarrow} Y).
\end{equation*}

Suppose that $\phi: X_{1} \longrightarrow X \times_{Y} Y_{1}$ is an element of $Hom_{\tau^{-1}(X)}(X_{1} \overset{f_{1}}{\longrightarrow} X, \tau^{-1}(\sigma)(Y_{1} \overset{g_{1}}{\longrightarrow} Y))$. Then $\beta$ is defined by the post-composition

\begin{equation*}
\beta: \phi \mapsto p^{*}_{g_{1}}(\sigma) \circ \phi
\end{equation*}
It's straightforward to see that all necessary triangles commute and the map is correctly defined.

Now, a map $\delta$, that goes in the converse direction, will be defined. Suppose we are given a morphism $\alpha: X_{1} \longrightarrow Y_{1}$ that is an element of \begin{center}
$Hom_{\tau^{-1}(Y)}(\overline{\sigma}(X_{1} \overset{f_{1}}{\longrightarrow} X), Y_{1} \overset{g_{1}}{\longrightarrow} Y)$.
\end{center}
Then, we have the universal morphism $\alpha'$ and the commutative diagram

\begin{equation}
\begin{gathered}
\xymatrix@R=3pc@C=3pc{
  X_{1} \ar@/^2.5pc/[drr]^{\alpha} \ar@/_1pc/[ddr]_-{f_{1}} \ar@{.>}[dr]^-{\alpha'} & &   \\
  & X \times_{Y} Y_{1} \ar[r]^-{p^*_{g_{1}}({\sigma})} \ar[d]^{p^*_{\sigma}(g_{1})} &  Y_{1} \ar[d]_{g_{1}} \\
  & X \ar[r]^{\sigma} & Y                                                                     }
\end{gathered}
\end{equation}

The map $\delta$ is defined by the rule

\begin{equation*}
\delta: \alpha \mapsto \alpha'
\end{equation*}

It's easy to see that $\beta \circ \delta$ and $\delta \circ \beta$ are identity maps. It's also quite straightforward to show that defined maps are natural: one should just keep in mind the uniqueness of the universal morphism.

\vspace{5mm}

(ii). The proof goes as in (i).
\end{proof}

\begin{proposition}

Suppose that $\mc{C}$ is a balanced category\footnote{Recall that a category $\mc{C}$ is called balanced if every morphism that is an epimorphism and monomorphism is also an isomorphism. Among examples of balanced categories are categories of groups, $R$-modules, rings, abelian categories.}. Then $(X, X \overset{\sigma}{\longrightarrow} F(X))$ is a fixed point if and only if functors $\tau^{-1}(\sigma): \tau^{-1}(F(X)) \longrightarrow \tau^{-1}(X)$ and $s^{-1}(\sigma): s^{-1}(X) \longrightarrow s^{-1}(F(X))$ are equivalences of categories.
\end{proposition}

\begin{proof}

\underline{1}. If $\sigma: X \longrightarrow F(X)$ is an isomorphism, functors $\tau^{-1}(\sigma)$ and $s^{-1}(\sigma)$ are equivalences by Proposition 4.1.

\underline{2.1}. Suppose that $G: \tau^{-1}(X) \longrightarrow \tau^{-1}(F(X))$ is quasi-inverse to $\tau^{-1}(\sigma)$, i.e. there are isomorphisms of functors

\begin{equation}
\phi_{1}: G \circ \tau^{-1}(\sigma) \widetilde{\longrightarrow} Id_{\tau^{-1}(F(X))}
\end{equation}
\begin{equation}
\phi_{2}: \tau^{-1}(\sigma) \circ G \widetilde{\longrightarrow} Id_{\tau^{-1}(X)}
\end{equation}

Since $\tau^{-1}(\sigma)$ is right adjoint to $\overline{\sigma}$, we have a unit morphism

\begin{equation}
\phi_{3}: Id_{\tau^{-1}(X)} \longrightarrow \tau^{-1}(\sigma) \circ \overline{\sigma}
\end{equation}

From (4.8) follows that there is a morphism $G \overset{\phi_{3}'}{\xrightarrow{\hspace*{1cm}}} G \circ \tau^{-1}(\sigma) \circ \overline{\sigma}$, and from (4.6) we have a morphism $G \circ \tau^{-1}(\sigma) \circ \overline{\sigma} \overset{\phi_{1}'}{\xrightarrow{\hspace*{1cm}}} \overline{\sigma}$. Hence, the morphism from $G$ to $\overline{\sigma}$ is defined, and we denote it as
\begin{equation}
\phi_{31}: G \longrightarrow \overline{\sigma}
\end{equation}

Now, suppose that $H: s^{-1}(F(X)) \longrightarrow s^{-1}(X)$ is a quasi-inverse to $s^{-1}(\sigma)$. We have two isomorphisms of functors

\begin{equation}
\phi_{4}: H \circ s^{-1}(\sigma) \widetilde{\longrightarrow} Id_{s^{-1}(X)}
\end{equation}

\begin{equation}
\phi_{5}: s^{-1}(\sigma) \circ H \widetilde{\longrightarrow} Id_{s^{-1}(F(X))}
\end{equation}

Since $s^{-1}(\sigma)$ is left adjoint to $\tilde{\sigma}$, there is a counit morphism
\begin{equation}
\phi_{6}: s^{-1}(\sigma) \circ \tilde{\sigma} \widetilde{\longrightarrow} Id_{s^{-1}(F(X))}
\end{equation}

From (4.10) follows that the morphism $\tilde{\sigma} \overset{\phi_{4}'^{-1}}{\xrightarrow{\hspace*{1cm}}} H \circ s^{-1}(\sigma) \circ \tilde{\sigma}$ is defined. Using (4.12), we build the morphism $H \circ s^{-1}(\sigma) \circ \tilde{\sigma} \overset{\phi_{6}'}{\xrightarrow{\hspace*{1cm}}} H$. Finally, there is the morphism of functors

\begin{equation}
\phi_{46}: \tilde{\sigma} \longrightarrow H
\end{equation}

\underline{2.2}. Since $G$ is quasi-inverse to $\tau^{-1}(\sigma)$, there is an isomorphism for arbitrary $g_{1}: Y_{1} \longrightarrow F(X) \in \mc{C}/F(X)$

\begin{equation}
 G \circ \tau^{-1}(\sigma) (Y_{1} \overset{g_1}{\longrightarrow} F(X)) \widetilde{\longrightarrow} (Y_{1} \overset{g_1}{\longrightarrow} F(X))
\end{equation}

By definition of $\tau^{-1}(\sigma)$, we have that $\tau^{-1}(\sigma)(Y_{1} \overset{g_1}{\longrightarrow} F(X)) = (X \times_{F(X)} Y_{1} \overset{p^*_{\sigma}(g_{1})}{\xrightarrow{\hspace*{1cm}}} X)$. Then we apply $G$ and denote the resulting image of $g_{1}$ under $G \circ \tau^{-1}(\sigma)$ as $G(X \times_{F(X)} Y_{1}) \overset{G(p^*_{\sigma}(g_{1}))}{\xrightarrow{\hspace*{1cm}}} F(X)$.

From (4.14) we see that there is an isomorphism $\phi$ in $\mc{C}$ such that the triangle

\begin{equation}
\begin{gathered}
\xymatrix@R=3pc@C=3pc{
  G(X \times_{F(X)} Y_{1}) \ar[r]^-{G(p^*_{\sigma}(g_{1}))} \ar[d]_{\phi}  &  F(X) \\
  Y_{1}  \ar[ur]^{g_{1}} & }
\end{gathered}
\end{equation}
is commutative.

Since morphism of functors $\phi_{31}: G \longrightarrow \overline{\sigma}$ is defined (see (4.9)), there is a component of $\phi_{31}$ associated with object $X \times_{F(X)} Y_{1} \overset{p^*_{\sigma}(g_{1})}{\xrightarrow{\hspace*{1cm}}} X$ of $\mc{C}/X$, and, therefore, we have a morphism $\eta$ in $\mc{C}$ that makes the diagram

\begin{equation}
\begin{gathered}
\xymatrix@R=3pc@C=3pc{
  G(X \times_{F(X)} Y_{1}) \ar[d]_{\eta} \ar[rr]^-{G(p^*_{\sigma}(g_{1}))} & & F(X)  \\
  X \times_{F(X)} Y_{1} \ar[r]^-{p^*_{\sigma}(g_{1})} & X \ar[ur]_{\sigma}  &     }
\end{gathered}
\end{equation}
a commutative one.

We're going to exploit universal property of pullback by using morphisms $\phi$ and $\eta$. Indeed, from (4.15) and (4.16) it's clear that there is a commutative diagram

\begin{equation}
\begin{gathered}
\xymatrix@R=3pc@C=3pc{
  G(X \times_{F(X)} Y_{1}) \ar@/^2.5pc/[drr]^{\phi} \ar@/_1pc/[ddr]_-{p^*_{\sigma}(g_{1}) \circ \eta} \ar@{.>}[dr]^-{\Phi} & &   \\
  & X \times_{F(X)} Y_{1} \ar[r]^-{p^*_{g_{1}}({\sigma})} \ar[d]^{p^*_{\sigma}(g_{1})} &  Y_{1} \ar[d]_{g_{1}} \\
  & X \ar[r]^{\sigma} & F(X)                                                                     }
\end{gathered}
\end{equation}
where $\Phi$ is universal morphism. From the right triangle in the diagram (4.17) we get the relation

\begin{center}

$p^*_{g_{1}}({\sigma}) \circ \Phi = \phi$,
\end{center}
where $\phi$ is an isomorphism as follows from (4.15). Hence, $p^*_{g_{1}}({\sigma})$ is a split epimorphism.

Since $g_{1}$ is an arbitrary object of $\mc{C}/F(X)$ we may choose $g_{1} = Id_{F(X)}$ and the diagram (4.17) still commutes. In this case, morphism $p^*_{\sigma}(g_{1})$ is an isomorphism. Then, from the commutativity of the pullback square follows that $\sigma$ is a split epimorphism.

\underline{2.3}. The same arguments apply to the dual construction of under-categories $X/\mc{C}$ and $F(X)/\mc{C}$. We only give the diagram that follows from the universal property of pushout

\begin{equation}
\begin{gathered}
\xymatrix@R=3pc@C=3pc{
  & & H(X_{1} \bigsqcup_{X} F(X))  \\
  X_{1} \ar@/^2.5pc/[urr]^{\phi'} \ar[r]^-{p_{*}^{f_{1}}(\sigma)} & X_{1} \bigsqcup_{X} F(X) \ar@{.>}[ur]^{\Psi} &   \\
  X \ar[r]^{\sigma} \ar[u]^{f_{1}} & F(X) \ar[u]^{p_{*}^{\sigma}(f_{1})} \ar@/_2.5pc/[uur]_{\eta' \circ p_{*}^{\sigma}(f_{1})}  &                                                                   }
\end{gathered}
\end{equation}
where isomorphism $\phi'$ is a component of the natural transformation (4.10) and $\eta'$ is a component of the natural transformation (4.13). From the same reasons as in dual construction, $\sigma$ is a split monomorphism.

\vspace{5mm}

Since category $\mc{C}$ is balanced according to the hypothesis, $\sigma$ is an isomorphism.

\end{proof}

\section{Fixed points of a site morphism}

In this section, we examine situation when category $\mc{C}$ is a site\footnote{We note again that what we call a site was originally called a category with pretopology; see the definitions in chapter 3.} and functor $F: \mc{C} \longrightarrow \mc{C}$ is a morphism of sites. Throughout this section we suppose that category $\mc{C}$ has pullbacks, and we also note that $\mc{C}$ is supposed to be small.

Hereinafter, a sheaf on $\mc{C}$ is always a sheaf of abelian groups.

\begin{proposition}
Suppose that $F: \mc{C} \longrightarrow \mc{C}$ is a morphism of sites. If $\mu$ is a sheaf on $\mc{C}$ then $\mu \circ F$ is a sheaf on $\mc{C}$.
\end{proposition}

\begin{proof}
It's clear that if $\nu$ is a presheaf on $\mc{C}$ then so is $\nu \circ F$. Since $\mu$ is a sheaf, the sequence

$$\mu(X) \longrightarrow \prod_{i \in A}\mu(X_{i}) \rightrightarrows \prod_{i,j \in A}\mu(X_{i} \times_{X} X_{j})$$
is an equalizer for all $X \in Ob(\mc{C})$ and all sets of indexes $A$ such that $\{X_{i} \longrightarrow X\}_{i \in A} \in cov(X)$. We should verify the same property for every sequence
\begin{equation}
\mu \circ F(X) \longrightarrow \prod_{i \in A}\mu \circ F(X_{i}) \rightrightarrows \prod_{i,j \in A}\mu \circ F(X_{i} \times_{X} X_{j})
\end{equation}

Since in the category of abelian groups direct product of isomorphisms is an isomorphism, and functor $F$, being a morphism of sites, commutes with pullbacks, we have an isomorphism

\begin{equation}
\prod_{i,j \in A}\mu \circ F(X_{i} \times_{X} X_{j}) \widetilde{\longrightarrow} \prod_{i,j \in A}\mu(F(X_{i}) \times_{F(X)} F(X_{j})).
\end{equation}

Composing sequence (5.1) on the right with (5.2), we get the sequence

\begin{equation}
\mu \circ F(X) \longrightarrow \prod_{i \in A}\mu \circ F(X_{i}) \rightrightarrows \prod_{i,j \in A}\mu(F(X_{i}) \times_{F(X)} F(X_{j})),
\end{equation}
which is an equalizer for all $X \in \mc{C}$ and all sets of indexes $A$. Indeed, $\mu$ is a sheaf and $F(\mc{U}) \in cov(F(X))$ for all $\mc{U} \in cov(X)$. Since (5.3) is an equalizer, the sequence (5.1) is always an equalizer.
\end{proof}

We define functor $F^{*}: Sh(\mc{C}) \longrightarrow Sh(\mc{C})$ by the pre-composition

$$\mu \mapsto \mu \circ F.$$

\begin{proposition}
Functor $F^{*}: Sh(\mc{C}) \longrightarrow Sh(\mc{C})$ is exact.
\end{proposition}

\begin{proof}
Functor $F$ between additive categories is additive if $F(0) \simeq 0$ and the canonical map $F(X \oplus Y) \longrightarrow F(X) \oplus F(Y)$ is an isomorphism. It's clear that $F^{*}$ is the additive one; we have equalities $\mu \circ F = 0$, if $\mu = 0$, and $(\mu \oplus \nu) \circ F =  \mu \circ F \oplus \nu \circ F$. Suppose that the sequence of sheaves

$$0 \longrightarrow \mu_{1} \overset{f_{1}}{\longrightarrow} \mu_{2} \overset{f_{2}}{\longrightarrow} \mu_{3} \longrightarrow 0$$
is exact. Then, we have to show that the sequence of sheaves

$$0 \longrightarrow \mu_{1} \circ F \overset{F^{*}(f_{1})}{\xrightarrow{\hspace*{1cm}}} \mu_{2} \circ F \overset{F^{*}(f_{2})}{\xrightarrow{\hspace*{1cm}}} \mu_{3} \circ F \longrightarrow 0$$
is also exact. But this follows immediately, since $\mu \circ F$ is nothing else but the restriction of $\mu$ to the subcategory of $\mc{C}$.
\end{proof}

Recall that a sheaf $\mu$ on $\mc{C}$ is called flabby if

$$H^{n}(\mc{U}, \mu) \simeq 0,$$
for all $\mc{U} \in cov(X)$, $X \in \mc{C}$, $n > 0$. Here, $H^{n}(\mc{U}, \mu)$ is given by $R^{n}H^{0}(\mc{U}, \mu)$, where

$$H^{0}(\mc{U}, \mu) = ker(\prod_{i}\mu(X_{i}) \rightrightarrows \prod_{i,j}\mu(X_{i} \times_{X} X_{j})).$$

Every injective sheaf is flabby, but the converse doesn't hold in general.

\begin{remark}
Groups $H^{n}(\mc{U}, \mu)$, where $\mc{U} = (X_{i} \longrightarrow X)_{i \in A} \in cov(X)$, may be computed as cohomology groups of the Cech complex $C^{\bullet}(\mc{U}, \mu)$ (see \cite{t}). By definition,

$$C^{q}(\mc{U}, \mu) = \prod_{(i_{0},...,i_{q}) \in A^{q+1}}\mu(X_{i_{0}} \times_{X} ... \times_{X} X_{i_{q}}),$$
while differential $d^{q}:  C^{q}(\mc{U}, \mu) \longrightarrow C^{q+1}(\mc{U}, \mu)$ is defined by

$$(d^{q}s)_{i_{0},...,i_{q+1}} = \sum_{k=0}^{q+1}(-1)^{k}\mu(\hat{k})(s_{i_{0},...,\hat{i}_{k},...,i_{q+1}}),$$
where $\hat{k}$ is projection

$$\hat{k} = X_{i_{0}} \times_{X} ... \times_{X} X_{i_{q+1}} \longrightarrow X_{i_{0}} \times_{X} ... \times_{X} \hat{X}_{i_{k}} \times_{X} ... \times_{X} X_{i_{q+1}}$$
\end{remark}

\begin{proposition}
If $\mu \in Sh(\mc{C})$ is a flabby sheaf, then $\mu \circ F$ is also a flabby sheaf.
\end{proposition}
\begin{proof}
For every $q > 0$ there is an isomorphism

$$\phi_{q}: C^{q}(\mc{U}, \mu \circ F) \widetilde{\longrightarrow} C^{q}(\mc{V}, \mu),$$
where $\mc{V} = F(\mc{U})$. Indeed, functor $F$ is a site morphism; therefore, it commutes with pullbacks and $F(\mc{U}) \in cov(F(X))$ for all $\mc{U} \in cov(X)$. For each $i > 0$ the diagram

\begin{equation*}
\begin{gathered}
\xymatrix@R=3pc@C=3pc{
  C^{i}(\mc{U}, \mu \circ F) \ar[d]_{\phi_{i}} \ar[r]^-{d^{i}} & C^{i+1}(\mc{U}, \mu \circ F) \ar[d]_{\phi_{i+1}} \\
  C^{i}(\mc{V}, \mu) \ar[r]^-{d^{i}} & C^{i+1}(\mc{V}, \mu)    }
\end{gathered}
\end{equation*}

is a commutative one. Hence,

$$H^{n}(\mc{U}, \mu \circ F) \simeq H^{n}(\mc{V}, \mu),$$ for all $n >0$.

From this isomorphism follows that if $\mu$ is a flabby sheaf then so is $\mu \circ F$.

\end{proof}

\begin{proposition}
Suppose that $(X,\alpha)$ is a fixed point of F, $\mu$ a sheaf on $\mc{C}$. For all $n$, there is an isomorphism

$$H^{n}(X, \mu) \simeq H^{n}(X, \mu \circ F).$$
\end{proposition}

\begin{proof}

Let $\mu^{\bullet}$ be an injective resolution of $\mu$. From Proposition 5.2 and Proposition 5.4 follows that $\mu^{\bullet} \circ F$ is a flabby resolution of $\mu \circ F$

$$\mu^{\bullet} = 0 \longrightarrow \mu_{1} \overset{f_{1}}{\longrightarrow} \mu_{2} \overset{f_{2}}{\longrightarrow} ...$$

$$\mu^{\bullet} \circ F = 0 \longrightarrow \mu_{1} \circ F \overset{F^{*}f_{1}}{\xrightarrow{\hspace*{1cm}}} \mu_{2} \circ F \overset{F^{*}f_{2}}{\xrightarrow{\hspace*{1cm}}} ...$$

Applying left exact functor $\Gamma(X,\cdot)$ to these sequences we get two complexes of abelian groups

$$0 \longrightarrow \mu_{1}(X) \overset{(f_{1})_{X}}{\xrightarrow{\hspace*{1cm}}} \mu_{2}(X) \overset{(f_{2})_{X}}{\xrightarrow{\hspace*{1cm}}} ...$$

$$0 \longrightarrow \mu_{1} \circ F(X) \overset{(F^{*}f_{1})_{X}}{\xrightarrow{\hspace*{1cm}}} \mu_{2} \circ F(X) \overset{(F^{*}f_{2})_{X}}{\xrightarrow{\hspace*{1cm}}} ...$$

Note that $(F^{*}f_{i})_{X} = (f_{i})_{F(X)}$. Since $(X,\alpha)$ is a fixed point and $f_{i}$ is a morphism of sheaves, we have, for each $i$, a commutative diagram

\begin{equation*}
\begin{gathered}
\xymatrix@R=3pc@C=3pc{
  \mu_{i}(X) \ar[d]_{\mu_{i}(\alpha)} \ar[r]^-{(f_{i})_{X}} & \mu_{i+1}(X) \ar[d]_{\mu_{i+1}(\alpha)} \\
  \mu_{i} \circ F(X) \ar[r]^-{(f_{i})_{F(X)}} & \mu_{i+1} \circ F(X)    }
\end{gathered}
\end{equation*}

Therefore, there is an isomorphism $H^{n}(X, \mu) \simeq H^{n}(X, \mu \circ F)$ for each $n$.

\end{proof}

In chapter 2, we introduced two classes of fixed points using an identity and an isomorphism. Now, we're going to set one more class of fixed points of a site morphism, using a relation that is weaker than an identity and an isomorphism.

\begin{definition}
A pair $(X, \Phi)$, where $X \in \mc{C}$ and $\Phi$ is the set of natural transformations $\Phi^{i}: H^{i}(X, \cdot) \longrightarrow H^{i}(F(X), \cdot), i \geq 0$, is called a cohomological fixed point of a site morphism $F: \mc{C} \longrightarrow \mc{C}$ if $\Phi^{i}$ is a natural isomorphism for all $i \geq 0$.
\end{definition}

\begin{proposition}
\begin{enumerate}[label=(\roman*)]
\item If $(X, \Phi)$ is a cohomological fixed point of $F: \mc{C} \longrightarrow \mc{C}$ then
$$H^{n}(X, \mu) \simeq H^{n}(X, \mu \circ F),$$
for all $\mu \in Sh(\mc{C}), n \geq 0$.

\item If for some $X \in \mc{C}$ there is an isomorphism $H^{n}(X, \mu) \simeq H^{n}(X, \mu \circ F)$, natural in $\mu$ for all $\mu \in Sh(\mc{C})$ and for all $n \geq 0$, then there is a cohomological fixed point of $F$.
\end{enumerate}
\end{proposition}
\begin{proof}
(i). By hypothesis, for each $i \geq 0$ there is an isomorphism $H^{i}(X, \mu) \widetilde{\longrightarrow} H^{i}(F(X), \mu)$, natural in $\mu$. But from Proposition 5.2 and Proposition 5.4 follows that $H^{i}(X, \mu \circ F) \widetilde{\longrightarrow} H^{i}(F(X), \mu)$, for all $i \geq 0$.

(ii). Again there is an isomorphism $H^{i}(X, \mu \circ F) \widetilde{\longrightarrow} H^{i}(F(X), \mu)$, natural in $\mu$ for each $i \geq 0$, as follows from Proposition 5.2 and Proposition 5.4. Since $H^{n}(X, \mu) \simeq H^{n}(X, \mu \circ F)$, for each $i \geq 0$, by the hypothesis, there is a family $\Phi$ of isomorphisms $\phi_{i}: H^{i}(X, \mu) \widetilde{\longrightarrow} H^{i}(F(X), \mu)$, for every $i \geq 0, \mu \in Sh(\mc{C})$, that are natural in $\mu$. Therefore, we have a cohomological fixed point of $F$, namely $(X,\Phi)$.

\end{proof}

\end{document}